\documentclass[10pt]{amsart}
\usepackage{latexsym,amsmath,amssymb,amscd}
\def\today{\ifcase \month \or
   January \or February \or March \or April \or
   May \or June \or July \or August \or
   September \or October \or November \or December \fi
   \space\number\day , \number\year}

\begin{document}

\DeclareRobustCommand{\SkipTocEntry}[4]{} 

\makeatletter
\@addtoreset{figure}{section}
\def\thefigure{\thesection.\@arabic\c@figure}
\def\fps@figure{h,t}
\@addtoreset{table}{bsection}

\def\thetable{\thesection.\@arabic\c@table}
\def\fps@table{h, t}
\@addtoreset{equation}{section}
\def\theequation{
\arabic{equation}}
\makeatother

\newcommand{\bfi}{\bfseries\itshape}

\newtheorem{theorem}{Theorem}
\newtheorem{acknowledgment}[theorem]{Acknowledgment}
\newtheorem{algorithm}[theorem]{Algorithm}
\newtheorem{axiom}[theorem]{Axiom}
\newtheorem{case}[theorem]{Case}
\newtheorem{claim}[theorem]{Claim}
\newtheorem{conclusion}[theorem]{Conclusion}
\newtheorem{condition}[theorem]{Condition}
\newtheorem{conjecture}[theorem]{Conjecture}
\newtheorem{construction}[theorem]{Construction}
\newtheorem{corollary}[theorem]{Corollary}
\newtheorem{criterion}[theorem]{Criterion}
\newtheorem{definition}[theorem]{Definition}
\newtheorem{example}[theorem]{Example}
\newtheorem{lemma}[theorem]{Lemma}
\newtheorem{notation}[theorem]{Notation}
\newtheorem{problem}[theorem]{Problem}
\newtheorem{proposition}[theorem]{Proposition}
\newtheorem{remark}[theorem]{Remark}
\numberwithin{theorem}{section}
\numberwithin{equation}{section}

\newcommand{\todo}[1]{\vspace{5 mm}\par \noindent
\framebox{\begin{minipage}[c]{0.95 \textwidth}
\tt #1 \end{minipage}}\vspace{5 mm}\par}

\newcommand{\1}{{\bf 1}}
\newcommand{\Ad}{{\rm Ad}}
\newcommand{\Alg}{{\rm Alg}\,}
\newcommand{\Aut}{{\rm Aut}\,}
\newcommand{\ad}{{\rm ad}}
\newcommand{\Ci}{{\mathcal C}^\infty}
\newcommand{\de}{{\rm d}}
\newcommand{\ee}{{\rm e}}
\newcommand{\ev}{{\rm ev}}
\newcommand{\fin}{{\rm fin}}
\newcommand{\id}{{\rm id}}
\newcommand{\ie}{{\rm i}}
\newcommand{\GL}{{\rm GL}}
\newcommand{\gl}{{{\mathfrak g}{\mathfrak l}}}
\newcommand{\Hom}{{\rm Hom}\,}
\newcommand{\Img}{{\rm Im}\,}
\newcommand{\Ker}{{\rm Ker}\,}
\newcommand{\lf}{{\rm l}}
\newcommand{\Lie}{\text{\bf L}}
\newcommand{\m}{\textbf{m}}
\newcommand{\OO}{{\rm O}}
\newcommand{\pr}{{\rm pr}}
\newcommand{\Ran}{{\rm Ran}\,}
\renewcommand{\Re}{{\rm Re}\,}
\newcommand{\rb}{\text{\bf r}}
\newcommand{\Sp}{{\rm Sp}}
\renewcommand{\sp}{{{\mathfrak s}{\mathfrak p}}}
\newcommand{\spa}{{\rm span}\,}
\newcommand{\Tr}{{\rm Tr}\,}

\newcommand{\G}{{\rm G}}
\newcommand{\U}{{\rm U}}

\newcommand{\Ac}{{\mathcal A}}
\newcommand{\Bc}{{\mathcal B}}
\newcommand{\Cc}{{\mathcal C}}
\newcommand{\Ec}{{\mathcal E}}
\newcommand{\Dc}{{\mathcal D}}
\newcommand{\Hc}{{\mathcal H}}
\newcommand{\Jc}{{\mathcal J}}
\renewcommand{\Mc}{{\mathcal M}}
\newcommand{\Nc}{{\mathcal N}}
\newcommand{\Oc}{{\mathcal O}}
\newcommand{\Pc}{{\mathcal P}}
\newcommand{\Tc}{{\mathcal T}}
\newcommand{\Vc}{{\mathcal V}}
\newcommand{\Uc}{{\mathcal U}}

\newcommand{\Bg}{{\mathfrak B}}
\newcommand{\Fg}{{\mathfrak F}}
\newcommand{\Gg}{{\mathfrak G}}
\newcommand{\Ig}{{\mathfrak I}}
\newcommand{\Jg}{{\mathfrak J}}
\newcommand{\Lg}{{\mathfrak L}}
\newcommand{\Pg}{{\mathfrak P}}
\newcommand{\Sg}{{\mathfrak S}}
\newcommand{\Xg}{{\mathfrak X}}
\newcommand{\Yg}{{\mathfrak Y}}
\newcommand{\Zg}{{\mathfrak Z}}

\newcommand{\ag}{{\mathfrak a}}
\newcommand{\bg}{{\mathfrak b}}
\newcommand{\dg}{{\mathfrak d}}
\renewcommand{\gg}{{\mathfrak g}}
\newcommand{\hg}{{\mathfrak h}}
\newcommand{\kg}{{\mathfrak k}}
\newcommand{\mg}{{\mathfrak m}}
\newcommand{\n}{{\mathfrak n}}
\newcommand{\og}{{\mathfrak o}}
\newcommand{\pg}{{\mathfrak p}}
\newcommand{\sg}{{\mathfrak s}}
\newcommand{\tg}{{\mathfrak t}}
\newcommand{\ug}{{\mathfrak u}}

\makeatletter
\title[On the characterization of Gelfand-Shilov-Roumieu spaces]
{On the characterization of Gelfand-Shilov-Roumieu spaces}
\author{Mihai Pascu}
\address{Institute of Mathematics ``Simion Stoilow'' 
of the Romanian Academy, 
RO-014700 Bucharest, Romania and ``Petroleum-Gas'' University of Ploiesti, Bd. Bucuresti, 39, Ploiesti}
\email{mihai.pascu@imar.ro}
\thanks{This research has been partially supported by the Grant
of the Romanian National Authority for Scientific Research, CNCS-UEFISCDI,
project number PN-II-ID-PCE-2011-3-0131.}

\date{\today}
\makeatother

\begin{abstract} 
Generalized $\mathbf{m}$-Gelfand-Shilov-Roumieu vector spaces $\mathcal{S}_{\mathbf{m}}(\mathbf{X})$ are introduced. Here  $\mathbf{m}=(m^{(1)},...,m^{(n)})$, $\mathbf{X}=(X_{1},...,X_{n})$ and  $m^{(1)},...,m^{(n)}$ are sequences of positive real numbers and  $X_{1},...,X_{n}$ are operators in a Hilbert space. Our definition extends Ter Elst's definition of Gevrey vector spaces [TE2]. Conditions are given on the sequences $m^{(1)},...,m^{(n)}$ and on the operators $X_{1},...,X_{n}$ so that the equality ${S}_{\mathbf{m}}(\mathbf{X})={S}_{m^{(1)}}(X_{1})\cap...\cap{S}_{m^{(n)}}(X_{n})$ is valid. As a corollary we obtain a new proof of a characterization theorem for classical Gelfand-Shilov spaces.
\end{abstract}

\maketitle

\tableofcontents

\section{Introduction}\label{sectionIntr}

\par{The GS (Gelfand-Shilov) spaces were introduced in the last chapter of the second volume of I. M. Gelfand and G. E. Shilov's monograph on Generalized Functions [GS]. They are subspaces of Schwartz's space ([Sch]) $\mathcal{S}$ of rapidly decreasing functions. Usually a $C^{\infty}$ function $\varphi:\mathbb{R}^{n}\rightarrow \mathbb{C}$ is said to belong to a GS space $\mathcal{S}^{a}_{b}$, $a,b \in \mathbb{R}_{+}^{n}$ if there exist positive constants $C, A$ and $B$ so that}

\begin{equation} 
\vert x^{\alpha} \partial^{\beta}\varphi (x) \vert \leq CA^{\vert \alpha \vert} B^{\vert \beta \vert} (\alpha!)^{a} (\beta!)^{b}, \forall x \in \mathbb{R}^{n}, \forall \alpha, \beta \in \mathbb{N}^{n}. 
\end{equation}

\par{ We have used the following notations: $x^{\alpha}=x^{\alpha_{1}}_{1}...x^{\alpha_{n}}_{n}$, $\vert \alpha \vert = \alpha_{1}+...+\alpha_{n}$, $\alpha!^{a}=\alpha_{1}!^{a_{1}}...\alpha_{n}!^{a_{n}}$ for $\alpha = (\alpha_{1},...,\alpha_{n})$ and $a=(a_{1},...,a_{n})$. In the same monograph Gelfand and Shilov introduced more general spaces. Let $m^{(1)}$,...,$m^{(n)}$, $q^{(1)}$,..., $q^{(n)}$ be sequences of positive real numbers, $\mathbf{m}=(m^{(1)},...,m^{(n)})$, $\mathbf{q}=(q^{(1)},...,q^{(n)})$. A rapidly decaying function $\varphi:\mathbb{R}^{n}\rightarrow \mathbb{C}$ belongs to the space $\mathcal{S}^{\mathbf{q}}_{\mathbf{m}}$ if there exist positive constants $C, A$ and $B$ such that

\begin{equation}
\vert x^{\alpha} \partial^{\beta}\varphi (x) \vert \leq CA^{\vert \alpha \vert} B^{\vert \beta \vert} \mathbf{m}_{\alpha}\mathbf{q}_{\beta}, \forall x \in \mathbb{R}^{n}, \forall \alpha, \beta \in \mathbb{N}^{n},
\end{equation}
where $\mathbf{m}_{\alpha}=m_{\alpha_{1}}^{(1)}...m_{\alpha_{n}}^{(n)}$.}

\par{ These spaces were also thoroughly studied in Roumieu's PhD thesis ([R]). In fact, in [GS] and [R], only functions which depended on a single variable were considered, but the generalization of their definition to more variables is straightforward. In this paper we shall call the spaces of functions which satisfy (1.1) GS spaces of Gevrey type and the more general spaces of functions which satisfy (1.2) Gelfand-Shilov-Roumieu (GSR) spaces.}
\par{ A series of papers have been devoted to the study of these spaces and alternative characterizations have been provided. We can distinguish two approaches, a pure analytical one, see [Ka], [CCK], [GZ], [Te1], [Te2] for GS spaces of Gevrey type and [Ba], [CCK], [Pa1], [Pa2] for GSR spaces, and another one which involves more functional analytic and Lie groups theory methods, see [GW], [EG], [Ei], [TE1], [TE2] for GS spaces of Gevrey type. In this context, we have also to mention the seminal works of Nelson ([Ne]), Nussbaum ([Nu]) and Goodman ([Go]).}

\par{ In [Ka] Kashpirovski answered to a question raised in [GS] and proved the equality $\mathcal{S}^{b}_{a}=\mathcal{S}^{b}\cap \mathcal{S}_{a}$ for $n=1$, $a,b \geq 0$, where} 
\begin{center}
$\mathcal{S}_{a}=\left\{\varphi \in \mathcal{S} ; \forall \beta \in \mathbb{N}, \exists C=C(\beta), \exists A>0,  \vert x^{\alpha} \partial^{\beta}\varphi (x) \vert \leq CA^{ \alpha } \alpha!^{a},\forall x \in \mathbb{R}, \forall \alpha \in \mathbb{N} \right\}$,
\end{center}
and
\begin{center}
$\mathcal{S}^{b}=\left\{\varphi \in \mathcal{S}; \forall \alpha \in \mathbb{N}, (\exists) C=C(\alpha), \exists B>0,   \vert x^{\alpha} \partial^{\beta}\varphi (x) \vert \leq CB^{\beta} \beta!^{b}, \forall x \in \mathbb{R}, \forall \beta\in  \mathbb{N} \right\}$.
\end{center}

\par{In [CCK] a similar result was proved for GSR spaces. If $m^{(1)}=...=m^{(n)}=m$, $q^{(1)}=...=q^{(n)}=q$, if these sequences satisfy assumptions (A1), (A2) and an assumption slightly weaker than (A3) ((A1), (A2), and (A3) will be introduced in the next section), then a function $\varphi$ belongs to $\mathcal{S}^{\mathbf{q}}_{\mathbf{m}}$ if and only if there exist positive constants $C, A$ and $B$ such that}

\begin{center}
$\vert x^{\alpha} \varphi (x) \vert \leq CA^{\vert \alpha \vert}  m_{\vert \alpha \vert}$, $\vert  \partial^{\beta}\varphi (x) \vert \leq C B^{\vert \beta \vert} q_{\vert \beta \vert}$, $\forall x \in \mathbb{R}^{n}$, $\forall \alpha, \beta \in \mathbb{N}^{n}$.
\end{center}

\par{In [TE2], TerElst obtained Kashpirovsky's result as a corollary of a general theorem on Gevrey vectors relative to finite families of operators in a Hilbert space $\mathcal{H}$. For $n\in \mathbb{N}^{\ast}$ let $M(n)$ be the set of multiindices $\bigcup _{k=1}^{\infty} \left\{1,...,n \right\}^{k}$. If $\alpha \in M(n)$ is so that $\alpha = (i_{1},...,i_{k}) \in \left\{1,...,n \right\}^{k}$ we put $\vert \alpha \vert = k$ ($\vert \alpha \vert $ is the length of $\alpha$) and, for every $j \in \left\{1,...,n \right\}$, $\vert \alpha \vert _{j} = \text{card} \left \{ l \in \left\{1,...,n \right\} ; i_{l} =j \right\}$.
}
\par{If $X_{1},..., X_{n}$ are operators in $\mathcal{H}$ and if $u\in \mathcal{H}$, then $u$ is called a Gevrey vector of order  $\lambda =( \lambda_{1},...,\lambda_{n}) $  relative to $X_{1},..., X_{n}$ ($u \in \mathcal{S} _{\lambda_{1},...,\lambda{n}}(X_{1},...,X_{n})$) if $u$ belongs to the domain of the operator $X_{i_{1}}...X_{i_{k}}$, $\forall k \in \mathbb{N}$ and $\forall \alpha =( i_{1},...,i_{k}) \in M(n)$ and if there exist some positive constants $C$ and $A$ so that
\begin{center}
 $\Vert X_{i_{1}}...X_{i_{k}}u \Vert \leq CA^{k}\alpha!^{\lambda}, \forall k \in \mathbb{N}$.
\end{center}} 
 \par{Now $ \alpha!^{\lambda}= \vert \alpha \vert _{1}^{\lambda_{1}}...\vert \alpha \vert _{n}^{\lambda_{n}}$. Ter Elst investigated the problem of finding appropriate conditions on $\lambda_{1},...,\lambda_{n}$ and  $X_{1},..., X_{n}$ in order that $\mathcal{S} _{\lambda_{1},...,\lambda{n}}(X_{1},...,X_{n})=\bigcap^{n}_{j=1} \mathcal{S}_{\lambda_{j}}(X_{j})$.}
\par{As a corollary of his results one obtains Kashpirovsky's theorem.}
\par{ We prove that Ter Elst's approach works also for the more general spaces of Gelfand-Shilov-Roumieu vectors where $\alpha!^{\lambda}$ is replaced with $\mathbf{m}_{\alpha}$, if the sequences $m^{(1)},...,m^{(n)}$ satisfy assumptions (A1)-(A3) from the next section. Chung, Chung, Kim's result can be obtained as a corollary of this approach. Let us stress that other general results on the intersections of Gevrey spaces are valid for the more general spaces of Gelfand-Shilov-Roumieu vectors. We shall give an example in the third section.} 

\section{The Heisenberg group}\label{section Heis}

\par{First of all we have to introduce some more notations.}
\par{  If $\alpha = (i_{1},...,i_{k}),\beta= (j_{1},...,j_{l}) \in M(n)$, then $\alpha \vee \beta =(i_{1},...,i_{k},j_{1},...,j_{l})$ is their concatenation. The reverse of $\alpha=(i_{1},...,i_{k})$ is $\alpha^{r}=(i_{k},...,i_{1})$.}
\par{ If $m^{(j)}=(m^{(j)}_{p})_{p}, \forall j \in \left\{ 1,...,n \right\}$ are sequences of positive real numbers and $\alpha \in M(n)$, then we put $\mathbf{m}_{\alpha}=m^{(1)}_{\vert \alpha \vert _{1}} ...m^{(n)}_{\vert \alpha \vert _{n}}$, $\mathbf{m}'=(m^{(1)},...,m^{(n-1)})$. Also $\alpha' \in M(n-1)$ is the multiindex which is obtained from $\alpha = (i_{1},...,i_{k})$ after the elimination of its components $i_{l}$ with $i_{l}=n$ and $\mathbf{m}'_{\alpha'}=m^{(1)}_{\vert \alpha \vert _{1}} ...m^{(n-1)}_{\vert \alpha \vert _{n-1}}$ .}

\par{ For an operator $Y$ in a Hilbert space $\mathcal{H}$ we denote with $Dom(Y)$ its domain of definition and if $Y$ and $Z$ are operators in $\mathcal{H}$, $[Y,Z]$ denotes their commutator, whenever this can be defined. Let $X_{1},...,X_{n}$ be operators in $\mathcal{H}$, $\mathbf{X}=(X_{1},...,X_{n})$. If $\alpha=(i_{1},...,i_{k})$, then $\mathbf{X}_{\alpha}=X_{i_{1}}...X_{i_{k}}$. We shall also use the notation $\mathbf{X}'=(X_{1},...,X_{n-1})$. The space of $\mathbf{X}-$smooth vectors is

\begin{center}
$C^{\infty}(\mathbf{X})= \left\{ u\in \mathcal{H}; u \in Dom(\mathbf{X}_{\alpha}), \forall \alpha \in M(n) \right\}$.
\end{center}
}
\par{ The space of $\mathbf{m}-$Gelfand-Shilov-Roumieu(GSR) vectors relative to $\mathbf{X}$ is

\begin{center}
$\mathcal{S}_{\mathbf{m}}(\mathbf{X})=\left\{ u\in C^{\infty}(\mathbf{X}); \exists A,C>0 \text{ such that } \Vert \mathbf{X}_{\alpha}u \Vert \leq CA^{\vert \alpha \vert }\mathbf{m}_{\alpha}, \forall \alpha \in M(n) \right\}$.
\end{center}
}
\par{We shall impose the following conditions on the sequences $m^{(j)}$:}
\par{(A0) $m^{(j)}_{0}=m^{(j)}_{1}=1,\forall j \in \left\{ 1,...,n \right\}$;}
\par{(A1) $(m^{(j)}_{p})^{2} \leq m^{(j)}_{p-1}m^{(j)}_{p+1},\forall j \in \left\{ 1,...,n \right\}, \forall p \in \mathbb{N}^{\ast}$ (logarithmic convexity);}
\par{(A2) $\exists H>0$ so that $m^{(j)}_{p+q} \leq H^{p+q} m^{(j)}_{p}m^{(j)}_{q},\forall j \in \left\{ 1,...,n \right\}, \forall p,q \in \mathbb{N}$ (the ultradifferentiability condition of Komatsu ([K]).}

\begin{remark}
\normalfont
 If $m^{(j)}$ is a logarithmic convex sequence and if $i_{0}$ is so that $m^{(j)}_{i_{0}} \leq  m^{(j)}_{i_{0}+1}$, then $m^{(j)}_{i} \leq  m^{(j)}_{i+1}, \forall i \geq i_{0}$. Therefore if $m^{(j)}$ is an unbounded sequence, there exists $i_{0}$ so that $1 \leq m^{(j)}_{i} \leq  m^{(j)}_{i+1}, \forall i \geq i_{0}$. Since the definition of the spaces $\mathcal{S}_{\mathbf{m}}(\mathbf{X})$ does not depend on the values of a finite number of terms of the sequences $m^{(j)}$, the technical assumption (A0) is not a restrictive one in  this case. Also we can assume without any loss of generality that $m^{(j)}$ are nondecreasing sequences.
\end{remark}
\par{The fourth assumption which we shall need is in the spirit of Definition 2.2 from [CCK] and ensures that the GSR spaces, as they are defined in [GS], are not trivial. This assumption relates two sequences $m^{(j)}$  and $m^{(k)}$:
\par{
(A3) $\exists L\geq1$ so that $pm^{(j)}_{p-1}m^{(k)}_{p-1} \leq L m^{(j)}_{p}m^{(k)}_{p}, \forall p \geq1$.}
}
\begin{lemma}
If $m^{(j)}$ is logarithmic convex, then 
\begin{equation*}
\frac{m^{(j)}_{p}}{m^{(j)}_{q}} \leq \frac{m^{(j)}_{p+1}}{m^{(j)}_{q+1}} \text{ for } q<p.
\end{equation*}
\end{lemma}
\begin{proof}
If $m^{(j)}$ is logarithmic convex, then 
\begin{equation*}
\frac{m^{(j)}_{p}}{m^{(j)}_{p-1}}\leq \frac{m^{(j)}_{p+1}}{m^{(j)}_{p}},\forall p \in \mathbb{N}^{\ast}.
\end{equation*}
 
\par{Therefore 
\begin{equation*}
\frac{m^{(j)}_{p}}{m^{(j)}_{q}} = \frac{m^{(j)}_{p}}{m^{(j)}_{p-1}}...\frac{m^{(j)}_{q+1}}{m^{(j)}_{q}} \leq \frac{m^{(j)}_{p+1}}{m^{(j)}_{p}}...\frac{m^{(j)}_{q+2}}{m^{(j)}_{q+1}} = \frac{m^{(j)}_{p+1}}{m^{(j)}_{q+1}}.
\end{equation*}
}
\end{proof}

\par{ The next result is a generalization of Lemma 1 from [TE2].}

\begin{lemma}
Let $X_{1},...,X_{n}$ be Hermitian or skew-Hermitian operators in a Hilbert space $\mathcal{H}$. Let $\mathbf{m}=(m^{(1)},...,m^{(n)})$ be such that $m^{(j)}$ satisfy (A2) $\forall j \in \left\{ 1,...,n \right\}$. Then $u\in \mathcal{S}_{\mathbf{m}}(\mathbf{X})$ if and only if $u\in C^{\infty}(\mathbf{X})$ and there exist two positive constants $A$ and $C$ such that
\begin{equation}
\vert (\mathbf{X}_{\alpha}u,u) \vert \leq CA^{\vert \alpha \vert}\mathbf{m}_{\alpha},\forall \alpha \in M(n).
\end{equation}
\end{lemma}
\begin{proof}
If the estimates (2.1) hold for some $u\in C^{\infty}(\mathbf{X})$ and if $\alpha \in M(n)$, then

\begin{equation*}
\begin{split}
\Vert \mathbf{X}_{\alpha}u \Vert ^{2} &= (\mathbf{X}_{\alpha}u,\mathbf{X}_{\alpha}u)= \vert (\mathbf{X}_{\alpha^{r} \vee \alpha} u,u) \vert \leq CA^{2\vert \alpha \vert} \mathbf{m}_{\alpha^{r} \vee \alpha}= CA^{2\vert \alpha \vert} m^{(1)}_{2\vert \alpha \vert _{1}}...m^{(n)}_{2\vert \alpha \vert _{n}} \leq\\
&\leq CA^{2\vert \alpha \vert}H^{2\vert \alpha \vert} (m^{(1)}_{\vert \alpha \vert _{1}})^{2}...(m^{(n)}_{\vert \alpha \vert _{n}})^{2} = C(AH)^{2\vert \alpha \vert}(\mathbf{m}_{\alpha})^{2}.
\end{split}
\end{equation*}

\par{ If $u\in \mathcal{S}_{\mathbf{m}}(\mathbf{X})$, then (1) follows from Schwarz's inequality: $\vert (\mathbf{X}_{\alpha}u,u) \vert \leq \Vert \mathbf{X}_{\alpha}u \Vert \Vert u \Vert$.}
\end{proof}

\begin{proposition}
Let $X_{1},...,X_{n}$ be Hermitian or skew-Hermitian operators in a Hilbert space $\mathcal{H}$ defined on a common invariant domain $D$, $m^{(1)},...,m^{(n)}$ sequences of positive numbers which satisfy (A0)-(A2). If the pair $(m^{(j)},m^{(n)})$ satisfies (A3) and if $[X_{j},X_{n}]=c_{j}I$ for some $c_{j} \in \mathbb{C}$ for every $j \in \left\{ 1,...,n-1 \right\}$, then $u\in D$ belongs to $\mathcal{S}_{\mathbf{m}}(\mathbf{X})$ if and only if $u \in\mathcal{S}_{\mathbf{m}'}(\mathbf{X}')\cap \mathcal{S}_{m^{(n)}}(X_{n})$.
\end{proposition}

\begin{proof}
Let $ u \in \mathcal{S}_{\mathbf{m}'}(\mathbf{X}')\cap \mathcal{S}_{m^{(n)}}(X_{n}) \cap D$. Then there exist some constants $A,C>0$ so that
\begin{center}
$\Vert \mathbf{X}'_{\alpha'}u \Vert \leq CA^{\vert \alpha' \vert}\mathbf{m}'_{\alpha'}, \forall \alpha' \in M(n-1)$,
\end{center}
\begin{center}
$\Vert X_{n}^{l}u \Vert \leq CA^{l}m_{l}^{(n)},\forall l \in \mathbb{N}$.
\end{center}

\par{ We shall apply Lemma 2.3. We shall prove by induction on $\vert \gamma \vert$ that
\begin{equation}
\vert (X_{n}^{l}\mathbf{X}_{\gamma}\mathbf{X}'_{\alpha'}u,u) \vert \leq C^{2}(2nL)^{\vert \gamma \vert} A^{\vert \alpha' \vert + \vert \gamma \vert +l} \mathbf{m}'_{\alpha' \vee \gamma' }m^{(n)}_{l+ \vert \gamma \vert _{n}}, \forall l \in \mathbb{N}, \forall \alpha' \in M(n-1), \forall \gamma \in M(n),
\end{equation}
for $A \geq \text{max} \left\{\text{max} \left\{ \vert c_{j} \vert;j  \in \left\{ 1,...,n-1 \right\}\right\},1 \right\}$ and $L$ sufficiently large, so that the inequalities (A3) hold for all the pairs  $(m^{(j)},m^{(n)})$  for every $j \in \left\{ 1,...,n-1 \right\}$. Remark that if $l=0$, $\vert \alpha' \vert =0$, then (2.2) coincides with (2.1).}
\par{If $\vert \gamma \vert = 0$, then 
\begin{center}
$\vert (X_{n}^{l} \mathbf{X}'_{\alpha'}u,u) \vert = \vert ( \mathbf{X}'_{\alpha'}u,X_{n}^{l}u) \vert \leq \Vert \mathbf{X}'_{\alpha'}u \Vert \cdot \Vert X_{n}^{l}u \Vert \leq C^{2}A^{l+ \vert \alpha' \vert} \mathbf{m}'_{\alpha'} m^{(n)}_{l}$
\end{center}
and (2.2) is true in this case.}

\par{Let us assume that the estimate is true for all multiindices of length $k$ and let $\gamma$ be a multiindex of length $k+1$. There are two possibilities: $\vert \gamma' \vert \geq \vert \gamma \vert _{n}$ or $\vert \gamma' \vert < \vert \gamma \vert _{n}$. We consider first the case when $\vert \gamma' \vert \geq \vert \gamma \vert _{n}$.}
\par{ If $\gamma = (i_{1},...,i_{k+1})$ and $i_{k+1} \in \left\{ 1,...,n-1 \right\}$ we have nothing to prove. So we assume that $X_{\gamma}=X_{\delta}X_{j}X_{n}^{i}$ for some $\delta \in M(n)$, $j \in \left\{ 1,...,n-1 \right\}$ and  $i \in \mathbb{N}^{\ast}$. We have
\begin{center}
$[X_{j},X_{n}^{i}]=(i-1)c_{j}X_{n}^{i-1}$,
\end{center}
so
\begin{center}
$X_{j}X_{n}^{i}=X_{n}^{i}X_{j}+(i-1)c_{j}X_{n}^{i-1}$.
\end{center}
}
\par{Then
\begin{equation*}
\begin{split}
\vert (X_{n}^{l}\mathbf{X}_{\gamma} \mathbf{X}'_{\alpha'}u,u) \vert &\leq \vert (X_{n}^{l}\mathbf{X}_{\delta}X_{n}^{i}X_{j}\mathbf{X}'_{\alpha'}u,u) \vert + (i-1) \vert c_{j} \vert \text{ } \vert (X_{n}^{l}\mathbf{X}_{\delta}X_{n}^{i-1}\mathbf{X}'_{\alpha'}u,u) \vert \leq\\
&\leq C^{2}(2nL)^{\vert \gamma \vert -1} A^{l+\vert \gamma \vert -1+ \vert \alpha' \vert +1} \mathbf{m}'_{\delta'\vee (j) \vee \alpha'} m^{(n)}_{l+i+\vert \delta \vert _{n}}+\\
&+C^{2}(i-1) \vert c_{j} \vert (2nL)^{\vert \gamma \vert -2} A^{l+\vert \gamma \vert -2+ \vert \alpha' \vert} \mathbf{m}'_{\delta'\vee \alpha'} m^{(n)}_{l+i-1+\vert \delta \vert _{n}} \leq\\
&\leq C^{2}(2nL)^{\vert \gamma \vert -1} A^{l+\vert \gamma \vert + \vert \alpha' \vert}\mathbf{m}'_{\alpha' \vee \gamma'}m^{(n)}_{l+\vert \gamma \vert _{n}} +\\
&+C^{2}(i-1) \vert (2nL)^{\vert \gamma \vert -2} A^{l+\vert \gamma \vert -1+ \vert \alpha' \vert} \mathbf{m}'_{\delta'\vee \alpha'} m^{(n)}_{l+\vert \gamma \vert _{n}-1} .
\end{split}
\end{equation*}
}
\par{Since $\vert \gamma' \vert \geq \vert \gamma \vert _{n} \geq i$ and, consequently, $\vert \delta' \vert \geq i-1$, there exists $p\in \left\{1,...,n-1 \right\}$ so that $ \vert \delta' \vert _{p} \geq \frac{i-1}{n} \geq \left[ \frac{i-1}{n} \right] \equiv \nu$. Using Lemma 2.1 and (A3) for the pair $(m^{(p)},m^{(n)})$, we obtain
\begin{equation*}
\begin{split}
(i-1)m^{(p)}_{\vert \alpha' \vee \delta' \vert _{p}}m^{(n)}_{l+\vert \gamma \vert _{n}-1} &= \frac{i-1}{\nu}\nu m^{(p)}_{\nu}m^{(n)}_{\nu} \frac{m^{(p)}_{\vert \alpha' \vee \delta' \vert _{p}}}{m^{(p)}_{\nu}} \frac{m^{(n)}_{l+\vert \gamma \vert _{n}-1}}{m^{(n)}_{\nu}} \leq\\
&\leq L\frac{i-1}{\nu} m^{(p)}_{\nu+1}m^{(n)}_{\nu+1} \frac{m^{(p)}_{\vert \alpha' \vee \delta' \vert _{p}+1}}{m^{(p)}_{\nu+1}} \frac{m^{(n)}_{l+\vert \gamma \vert _{n}}}{m^{(n)}_{\nu+1}}  \leq\\
&\leq 2nL m^{(p)}_{\vert \alpha' \vee \delta' \vert _{p}+1}m^{(n)}_{l+\vert \gamma \vert _{n}}
\end{split}
\end{equation*}
for $i\geq 1$.}
\par{Therefore
\begin{center}
$\vert (X_{n}^{l}\mathbf{X}_{\gamma} \mathbf{X}'_{\alpha'}u,u) \vert \leq C^{2}A^{l+\vert \gamma \vert + \vert \alpha' \vert} 2(2nL)^{\vert \gamma \vert -1} \mathbf{m}'_{\alpha'+\gamma'}m^{(n)}_{l+\vert \gamma \vert _{n}} \leq C^{2}A^{l+\vert \gamma \vert + \vert \alpha' \vert} (2nL)^{\vert \gamma \vert} \mathbf{m}'_{\alpha'+\gamma'}m^{(n)}_{l+\vert \gamma \vert _{n}}$,
\end{center}
and (2.2) holds for $\vert \gamma \vert =k+1$ in this case.}
\par{If $\vert \gamma' \vert < \vert \gamma \vert _{n}$, then we can write $\mathbf{X}_{\gamma}=X_{i_{1}}...X_{i_{\mu}}X_{n}\mathbf{X}_{\delta}$ for some $i_{1},...,i_{\mu} \in \left\{1,...,n-1 \right\}$, $\delta \in M(n)$. We have
\begin{center}
$[X_{i_{1}}...X_{i_{\mu}},X_{n}]=\sum_{\nu=1}^{\mu}c_{i_{\nu}}X_{i_{1}}...X_{i_{\nu-1}}X_{i_{\nu+1}}...X_{i_{\mu}}$.
\end{center}
}
\par{Using this formula and the induction hypothesis we can prove in a similar manner than (2.2) is valid in this case also.}

\end{proof}

\begin{remark}
\normalfont
If $X_{1},...,X_{n}$ are commuting Hermitian or skew-Hermitian operators in a Hilbert space $\mathcal{H}$ defined on a common invariant domain $D$ and if $m^{(1)},...,m^{(n)}$ are sequences of positive numbers which satisfy (A2), then $u \in D$ belongs to $\mathcal{S}_{\mathbf{m}}(\mathbf{X})$ if and only if $u \in \mathcal{S}_{\mathbf{m}'}(\mathbf{X}')\cap \mathcal{S}_{m^{(n)}}(X_{n})$.

\par{ Indeed, if $u \in \mathcal{S}_{\mathbf{m}'}(\mathbf{X}')\cap \mathcal{S}_{m^{(n)}}(X_{n}) \cap D$, $C$ and $A$ are positive constants as in the begining of the proof of Proposition 1 and if $\alpha \in M(n)$, then
\begin{center}
$\vert (\mathbf{X}_{\alpha}u,u) \vert =\vert ((X_{n}^{\vert \alpha \vert _{n}}\mathbf{X}'_{\alpha'}u,u) \vert \leq \Vert (X_{n}^{\vert \alpha \vert _{n}}u \Vert \text{ } \Vert \mathbf{X}'_{\alpha'} u \Vert \leq CA^{\vert \alpha' \vert} \mathbf{m}'_{\alpha'} \cdot CA^{\vert \alpha \vert _{n}}m^{(n)}_{\vert \alpha \vert _{n}} =C^{2}A^{\vert \alpha \vert}\mathbf{m}_{\alpha}$.
\end{center} 
}
\end{remark}

\begin{theorem}
Let $X_{1},...,X_{n}$ be Hermitian or skew-Hermitian operators in a Hilbert space $\mathcal{H}$ defined on a common invariant domain $D$, $[X_{i},X_{j}]=c_{ij}I$ for some constants $c_{ij} \in \mathbb{C}$, $ i,j \in \left\{ 1,...,n \right\}$. Let also $m^{(1)},...,m^{(n)}$ be sequences of positive numbers which satisfy (A0)-(A2) and also (A3) for all pairs $(i,j)$ with $c_{ij}\neq 0$. Then $u \in D$ belongs to $\mathcal{S}_{\mathbf{m}}(\mathbf{X})$ if and only if $u\in \mathcal{S}_{m^{(1)}}(X_{1})\cap...\cap\mathcal{S}_{m^{(n)}}(X_{n})$.
\end{theorem}

\begin{proof}
It is sufficient to prove that if $u \in  \mathcal{S}_{\mathbf{m}'}(\mathbf{X}')\cap \mathcal{S}_{m^{(n)}}(X_{n}) \cap D$ then $u \in \mathcal{S}_{\mathbf{m}}(\mathbf{X})$. We can proceed as in the proof of Proposition 1. We can prove by induction on $\vert \gamma \vert$ that if $u \in \mathcal{S}_{\mathbf{m}'}(\mathbf{X}')\cap \mathcal{S}_{m^{(n)}}(X_{n}) \cap D$, then $u$ satisfies (2.2). If $\gamma \in M(n)$, $\gamma = \delta \vee (j) \vee (n)^{i}$ and if $[X_{j},X_{n}]= 0$, then the hypothesis (A3) for the pair $(j,n)$ is not necessary.

\end{proof}

\begin{corollary}
Let $X_{1},...,X_{n}$ be Hermitian or skew-Hermitian operators in a Hilbert space $\mathcal{H}$ defined on a common invariant domain $D$,
\begin{equation*}
 [X_{i},X_{j}]=[Y_{i},Y_{j}]=0, \forall i,j \in \left\{i,...,n \right\},  [X_{i},Y_{j}]=0, \forall i,j \in \left\{i,...,n \right\}, i\neq j, [X_{i},Y_{i}]=c_{i}I, 
\end{equation*}
for some $c_{i} \in \mathbb{C} \setminus \left\{0 \right\}, \forall i \in \left\{i,...,n \right\}$.
 \par{ Let us also assume that $m^{(1)},...,m^{(n)},q^{(1)},...,q^{(n)}$ are sequences of positive numbers which satisfy (A0) - (A2) and that the pair $(m^{(i)},q^{(i)})$ satisfies (A3) $\forall i \in \left\{i,...,n \right\}$. If $u \in D$, then 
\begin{equation*} 
 u \in \mathcal{S}_{({m^{(1)},...,m^{(n)},q^{(1)},...,q^{(n)}})}(X_{1},...,X_{n},Y_{1},...,Y_{n})
 \end{equation*}
 if and only if 
 \begin{equation*}
 u \in \mathcal{S}_{\mathbf{m}}(\mathbf{X}) \cap \mathcal{S}_{\mathbf{q}}(\mathbf{Y})
 \end{equation*}
 and if and only if 
 \begin{equation*}
 u\in \left( \bigcap_{i=1}^{n} \mathcal{S}_{m^{i}}(X_{i}) \right) \bigcap \left( \bigcap_{i=1}^{n} \mathcal{S}_{q^{(i)}}(Y_{i})\right).
\end{equation*}}

\end{corollary}
\begin{remark}
\normalfont
Let us assume that $\pi$ is a strongly continuous representation of a Lie group $G$ on $\mathcal{H}$ and that $X_{j}$ is the infinitesimal generator of the strongly continuous one parameter group $t\mapsto \pi (\text{exp}tY_{j})$ for $j \in \left\{1,...,n \right\}$, where $\left\{ Y_{1},...,Y_{n} \right\}$ is a basis for the Lie algebra of the Lie group $G$. Then, accordingly to Theorem 1.1 from [Go], $C^{\infty}(\mathbf{X})= \bigcap_{j=1}^{n} C^{\infty}(X_{j})$. Therefore in this case $ \bigcap_{j=1}^{n} \mathcal{S}_{m^{(j)}}(X_{j}) \subset C^{\infty}(\mathbf{X})$ and, if the hypothesis of Theorem 2.6 is satisfied, then
\begin{center}
$\mathcal{S}_{\mathbf{m}}(\mathbf{X})=\mathcal{S}_{m^{(1)}}(X_{1})\cap...\cap\mathcal{S}_{m^{(n)}}(X_{n})$. 
\end{center}
\end{remark}

\par{ We shall apply now these results to the case when $X_{1},...,X_{n},Y_{1},...,Y_{n},Z$ , $X_{j}= \partial/\partial x_{j}, Y_{j} = x_{j}, \forall j \in \left\{1,...,n \right\}$, $Z=I$ are the infinitezimal generators of a representation of the Heisenberg algebra on $ \mathcal{H} = L^{2} (\mathbb{R}^{n})$.  In this case the hypothesis of Corollary 2.7 is satisfied. Hence, for a function $\varphi:\mathbb{R}^{n} \rightarrow \mathbb{C}$ there exists some positive constants $A, B$ and $C$ so that 

\begin{equation}
\Vert x^{\alpha} \partial^{\beta} \varphi \Vert \leq CA^{\vert \alpha \vert} B^{\vert \beta \vert} m_{\alpha} q_{\beta}, \forall \alpha, \beta \in \mathbb{N}^{n}
\end{equation}
if and only if there exists some positive constants $A, B$ and $C$ so that
\begin{equation}
\Vert x^{\alpha}  \varphi \Vert \leq CA^{\vert \alpha \vert}  m_{\alpha} ,\text{ }\Vert  \partial^{\beta} \varphi \Vert \leq C B^{\vert \beta \vert}  q_{\beta}, \forall \alpha, \beta \in \mathbb{N}^{n}. 
\end{equation}
}
\par{Now let us assume that $\varphi \in \mathcal{S}$ and 
\begin{equation}
\Vert x^{\alpha}  \varphi \Vert_{\infty} \leq C_{1}A_{1}^{\vert \alpha \vert}  m_{\alpha} ,\text{ }\Vert  \partial^{\beta} \varphi \Vert_{\infty} \leq C_{1}B_{1} ^{\vert \beta \vert}  q_{\beta}, \forall \alpha, \beta \in \mathbb{N}^{n},
\end{equation}
for some positive constants $A_{1},B_{1}$ and $C_{1}$.}
\par{Then 
\begin{center}
$\Vert x^{\alpha}  \varphi \Vert = \left(\int_{\mathbb{R}^{n}} \vert x^{\alpha}  \varphi(x) \vert ^{2} \text{ d}x \right)^{1/2} \leq \Vert (1+\vert x \vert ^{2})^{(1/2)([n/2]+1)} x^\alpha \varphi \Vert _{\infty} \left(\int_{\mathbb{R}^{n}} (1+\vert x \vert^{2}) ^{-[n/2]-1} \text{ d}x \right)^{1/2}$.
\end{center}
}
\par{ From (A2) we obtain that
\begin{center}
$\Vert x^{\alpha}  \varphi \Vert \leq C_{2} A_{1}^{\vert \alpha \vert + [n/4] +2} H^{\vert \alpha \vert + [n/4] +2}m_{\alpha} \leq C_{3} (A_{1}H)^{\vert \alpha \vert } m_{\alpha}, \forall \alpha \in \mathbb{N}^{n}$.
\end{center}
}
\par{On the other hand
\begin{equation*}
\begin{split}
\Vert \partial^{\beta} \varphi \Vert &=\left(\int_{\mathbb{R}^{n}} \partial^{\beta}  \varphi(x) \partial^{\beta}  \bar{\varphi}(x)\text{ d}x \right)^{1/2}=\vert \int_{\mathbb{R}^{n}} \varphi(x) \partial^{2\beta}  \bar{\varphi}(x)\text{ d}x \vert^{1/2} \leq \Vert \varphi \Vert _{L^{1}}^{1/2} \Vert \partial^{2\beta} \varphi \Vert _{\infty}^{1/2} \leq\\
&\leq \Vert \varphi \Vert _{L^{1}}^{1/2}(C_{1}B_{1}^{2\vert \beta \vert}m_{2\beta})^{1/2} \leq C_{4}(B_{1}H)^{\vert \beta \vert}m_{\beta}.
\end{split}
\end{equation*}
}
\par{ Therefore if $\varphi \in \mathcal{S}$ satisfies (2.5), then it satisfies also (2.4). We know also that $\varphi \in \mathcal{S}_{m}^{q}$ if and only if $\varphi$ satisfies (2.3) (see e. g. Lemma 2 and Lemma 3 from [Pa1]). Hence we have derived Chung-Chung-Kim's result: $\varphi \in \mathcal{S}_{m}^{q}$ if and only if $\varphi \in \mathcal{S}$ and $\varphi$ satisfies (2.5).}

\section{More general commutation relations}\label{sectionCenCom}
\par{The commutation relations between the operators $X_{j}$ we considered in the previous section are the commutation relations satisfied by the operators which generate the Heisenberg algebra. However, one can prove in a similar manner results for intersections of spaces of Gelfand-Shilov-Roumieu vectors if the generators of the Lie algebra of operators satisfy more general commutation relations. The simplest one is the following:
\begin{theorem}
Let $X_{1},...,X_{n}$ be Hermitian or skew-Hermitian operators in a Hilbert space $\mathcal{H}$ defined on a common invariant domain $D$, $m^{(1)}=...=m^{(n)}=m$, where $m$ is a sequence of positive numbers which satisfies (A0)-(A2) and (A3'): $p m_{p-1} \leq L m_{p}, \forall p \geq1$ for some constant $L \geq 1$. If $[X_{i},X_{j}] \in \text{span} (\left\{ X_{1},...,X_{n} \right\})$ for every $i,j \in \left\{ 1,...,n-1 \right\}$, then $u\in D$ belongs to $\mathcal{S}_{\mathbf{m}}(\mathbf{X})$ if and only if $u \in \cap_{j=1}^{n} \mathcal{S}_{m}(X_{j})$.
\end{theorem}
}
\par{Theorem 3.1 is an immediate consequence of Proposition 3.3 from below. In the proof of Proposition 3.3 we shall use the following remark.}
\begin{remark}
\normalfont
If $m^{(1)}=...=m^{(n)}=m$, where $m$ is a sequence of positive numbers which satisfies (A0)-(A2), then $u \in \mathcal{S}_{\mathbf{m}}(\mathbf{X})$ if and only if $ u\in C^{\infty}(\mathbf{X})$ and there exists $A,C>0$ such that $\Vert \mathbf{X}_{\alpha}u \Vert \leq CA^{\vert \alpha \vert}m_{\vert \alpha \vert}, \forall \alpha \in M(n) $. Therefore if $X_{1},...,X_{n}$ are Hermitian or skew-Hermitian operators in a Hilbert space $\mathcal{H}$, then $u\in \mathcal{S}_{\mathbf{m}}(\mathbf{X})$ if and only if $u\in C^{\infty}(\mathbf{X})$ and there exist two positive constants $A$ and $C$ so that
\begin{equation}
\vert (\mathbf{X}_{\alpha}u,u) \vert \leq CA^{\vert \alpha \vert}m_{\vert \alpha \vert},\forall \alpha \in M(n).
\end{equation}

\end{remark}

\begin{proposition}
Let $X_{1},...,X_{n}$ be Hermitian or skew-Hermitian operators in a Hilbert space $\mathcal{H}$ defined on a common invariant domain $D$, $m^{(1)}=...=m^{(n)}=m$, where $m$ is a sequence of positive numbers which satisfy (A0)-(A2) and (A3'). If $[X_{j},X_{n}]\in \text{span} (\left\{ X_{1},...,X_{n} \right\})$ for every $j \in \left\{ 1,...,n-1 \right\}$, then $u\in D$ belongs to $\mathcal{S}_{\mathbf{m}}(\mathbf{X})$ if and only if $u \in\mathcal{S}_{\mathbf{m}'}(\mathbf{X}')\cap \mathcal{S}_{m^{(n)}}(X_{n})$.
\end{proposition}

\begin{proof}The proof is similar to the proof of Proposition 2.4. So, let $[X_{j},X_{n}]= \sum_{p=1}^{n}{c_{j,p}X_{p}},\forall j \in \left\{ 1,...,n-1 \right\}$ and $L$ be the constant from the inequality (A3'). Let $A\geq 1$ and $C$ be constants as in the begining of the proof of Proposition 2.4. We may assume also that $A \geq \vert c_{j,p} \vert, \forall j \in \left\{1,...,n-1 \right\}, \forall p \in \left\{1,...,n \right\}$. We shall prove  by induction on $\vert \gamma \vert$ that 
\begin{equation}
\vert (X_{n}^{l}X_{\gamma}X'_{\alpha'}u,u) \vert \leq C^{2}((n+1)L)^{\vert \gamma \vert} A^{\vert \alpha' \vert + \vert \gamma \vert +l} m_{\vert \alpha' \vert + \vert \gamma \vert +l}, \forall l \in \mathbb{N}, \forall \alpha' \in M(n-1), \forall \gamma \in M(n).
\end{equation}
\par{Again, if $\vert \gamma \vert = 0$, then (3.2) is clearly true. We assume that the estimate is true for all multiindices of length $k$ and let $\gamma$ be a multindex of length $k+1$. If $\gamma = (i_{1},...,i_{k+1})$ and $i_{k+1} \in \left\{ 1,...,n-1 \right\}$ we have nothing to prove. So we assume that $\mathbf{X}_{\gamma}=\mathbf{X}_{\delta}X_{j}X_{n}^{i}$ for some $\delta \in M(n)$, $j \in \left\{ 1,...,n-1 \right\}$ and  $i \in \mathbb{N}^{\ast}$. We have
\begin{equation*}
[X_{j},X_{n}^{i}]=\sum_{q=0}^{i-1}\sum_{p=1}^{n} c_{j,p}X_{n}^{q}X_{p}X_{n}^{i-q-1}.
\end{equation*}
}
\par{Hence
\begin{equation*}
\begin{split}
\vert (X_{n}^{l}X_{\gamma} X'_{\alpha'}u,u) \vert &\leq \vert (X_{n}^{l}X_{\delta}X_{n}^{i}X_{j}X'_{\alpha'}u,u) \vert + \sum_{q=0}^{i-1}\sum_{p=1}^{n} \vert c_{j,p} \vert \text{ } \vert (X_{n}^{l}X_{\delta}X_{n}^{q}X_{p} X_{n}^{i-q-1}X'_{\alpha'}u,u) \vert \leq\\
&\leq C^{2}((n+1)L)^{\vert \gamma \vert -1} A^{l+\vert \gamma \vert + \vert \alpha' \vert } m_{\vert \alpha' \vert + \vert \gamma \vert +l} +\\
&+C^{2} ni \left(\text{ max} \left\{\vert c_{j,p} \vert ; p=1,...,n \right\} \right) ((n+1)L)^{\vert \gamma \vert -1} A^{l+\vert \gamma \vert + \vert \alpha' \vert -1 } m_{\vert \alpha' \vert + \vert \gamma \vert +l-1} \leq\\
&\leq C^{2}((n+1)L)^{\vert \gamma \vert -1} A^{l+\vert \gamma \vert + \vert \alpha' \vert } m_{\vert \alpha' \vert + \vert \gamma \vert +l} +\\
&+C^{2} nL  ((n+1)L)^{\vert \gamma \vert -1} A^{l+\vert \gamma \vert + \vert \alpha' \vert  } m_{\vert \alpha' \vert + \vert \gamma \vert +l} \leq\\
&\leq C^{2}((n+1)L)^{\vert \gamma \vert } A^{l+\vert \gamma \vert + \vert \alpha' \vert } m_{\vert \alpha' \vert + \vert \gamma \vert +l} .
\end{split}
\end{equation*}
}
\end{proof}
\par{In a future paper we intend to consider the case when the condition $m^{(1)}=...=m^{(n)}=m$ is not satisfied.}

\end{document}